\documentclass[12pt]{article}
\usepackage{amssymb, amsthm}
\usepackage{amsmath}

\newtheorem{theorem}{Theorem}
\newtheorem{Rem}{Remark}
\newtheorem{sats}{Theorem}
\newtheorem{prop}{Proposition}
\newtheorem{lem}{Lemma}
\newtheorem{kor}{Corollary}
\newcommand{\banm}{\begin{anm}}
\newcommand{\eanm}{\end{anm}}

\newcommand{\const}{{\rm const}}

\begin{document}

\title
{Oblique derivative problem \\
for  non-divergence parabolic equations\\
with discontinuous in time coefficients}
\author
{Vladimir Kozlov\footnote{Department of Mathematics, University of Link\"oping, SE-581 83 
Link\"oping, Sweden} \ and Alexander Nazarov\footnote{St.-Petersburg Department of Steklov 
Mathematical Institute, Fontanka, 27, St.-Petersburg, 191023, Russia, and  
St.-Petersburg State University, Universitetskii pr. 28, St.-Petersburg, 198504, Russia}
}

\date{}
\maketitle
\begin{abstract} 
\noindent We consider an oblique derivative problem for non-divergence parabolic
equations with discontinuous in $t$ coefficients in a half-space. We obtain
weighted coercive estimates of solutions in
anisotropic So\-bo\-lev spaces. We also give an application of this result to
linear 
 parabolic equations in a bounded domain. In particular, if the boundary is of class
${\cal C}^{1,\delta}$, $\delta\in (0,1]$, then we present a coercive estimate of solutions
in weighted anisotropic Sobolev spaces, where the weight is a power of the
distance to the boundary.
\end{abstract}

\section{Introduction}

Consider the parabolic equation
\begin{equation}\label{Jan1}
({\cal L}_0u)(x,t)\equiv\partial_tu(x,t)-
a^{ij}(t)D_iD_ju(x,t)=f(x,t)
\end{equation}
for $x\in\mathbb R^n$ and $t\in\mathbb R$.  Here and elsewhere $D_i$
denotes the operator of differentiation with respect to $x_i$ and
$\partial_tu$ is the derivative of $u$ with respect to $t$.

The only assumptions about the coefficients in (\ref{Jan1}) is that $a^{ij}$ are measurable real
valued functions of $t$ satisfying $a^{ij}=a^{ji}$ and
\begin{equation}\label{Jan2}
\nu|\xi|^2\le a^{ij}\xi_i\xi_j\le \nu^{-1}|\xi|^2, \qquad
\xi\in{\mathbb R}^n, \quad \nu=\const>0.
\end{equation}
It was proved by Krylov \cite{Kr2,Kr} that for $f\in L_{p,q}(\mathbb R^{n}\times\mathbb R)$ with
$1<p,q<\infty$, equation (\ref{Jan1}) in $\mathbb R^n\times\mathbb R$ has a unique solution such
that $\partial_tu$ and $D_iD_ju$ belong to $L_{p,q}(\mathbb R^n\times\mathbb R)$ and
\begin{equation}\label{Jan4}
\|\partial_tu\|_{p,q}+\sum_{ij} \|D_iD_ju\|_{p,q}\leq
C\|f\|_{p,q}\;.
\end{equation}
Here $L_{p,q}(\Omega\times I)=L_q(I\to L_p(\Omega))$ is the space of functions on
$\Omega\times I$  with finite norm
\begin{equation*}
\|f\|_{p,q}=\Big (\int\limits_I\Big (\int\limits_{\Omega}
|f(x,t)|^pdx\Big )^{\frac qp}dt\Big )^{\frac 1q}
\end{equation*}
(with natural change in the case $p=\infty$ or $q=\infty$).

In the authors' paper \cite{KN}  estimate (\ref{Jan4}) was
supplemented by a similar one in the space
$\widetilde{L}_{p,q}(\mathbb R^{n}\times\mathbb R)$
\begin{equation*}
|\!|\!|\partial_tu|\!|\!|_{p,q}+\sum_{ij}
|\!|\!|D_iD_ju|\!|\!|_{p,q}\leq C|\!|\!|f|\!|\!|_{p,q}.
\end{equation*}
Here $\widetilde{L}_{p,q}(\Omega\times I)=L_p(\Omega\to L_q(I))$ is the space of functions on
$\Omega\times I$  with finite norm
\begin{equation*}
|\!|\!|f|\!|\!|_{p,q}=\Big (\int\limits_\Omega\Big
(\int\limits_I |f(x,t)|^qdt\Big )^{\frac pq}dx\Big )^{\frac 1p}.
\end{equation*}
(with natural change in the case $p=\infty$ or $q=\infty$).
This space arises naturally in the theory of quasilinear non-divergence parabolic equations
(see \cite{Na}). Note that for $p=q$ we have
$$
\widetilde{L}_{p,p}(\Omega\times I)=L_{p,p}(\Omega\times I)=L_p(\Omega\times I);\qquad
|\!|\!|f|\!|\!|_{p,p}=\|f\|_{p,p}=\|f\|_p.
$$

The homogeneous Dirichlet problem for (\ref{Jan1}) in $\mathbb R^n_+\times\mathbb R$, where
$\mathbb R_+^{n}$ the half-space $\{x=(x',x_n)\in\mathbb R^n:x_n>0\}$, was considered
in \cite{Kr2,KN}. It was proved that its solution satisfies the
following weighted coercive estimate
\begin{equation}\label{Jan5}
\|x_n^\mu\partial_tu\|_{p,q}+\sum_{ij} \|x_n^\mu D_iD_ju\|_{p,q}\leq
C\|x_n^\mu f\|_{p,q}\;,
\end{equation}
where $1<p,q<\infty$ and $\mu\in\,(-\frac 1p,2-\frac 1p)$ (in \cite{Kr2}
this estimate was proved only for $\mu\in\,(1-\frac 1p,2-\frac 1p)$).
An analog of estimate (\ref{Jan5}), where the norm $\|\cdot \|_{p,q}$ is replaced by
$|\!|\!|\cdot|\!|\!|_{p,q}$, is also proved in \cite{KN}.\medskip

In the paper \cite{KN1} the homogeneous Dirichlet problem for (\ref{Jan1}) in cones and wedges
was considered, and coercive estimates for solutions were obtained in the scales of weighted 
$L_{p,q}$ and $\widetilde{L}_{p,q}$ spaces, where the weight is a power of the distance to the
vertex (edge).\medskip

Let us turn to  the oblique derivative problem  in the half-space
$\mathbb R^n_+$. Now equation (\ref{Jan1}) is satisfied for $x_n>0$ and
$\frac{\partial u}{\partial \gamma}=0$ for $x_n=0$. Here $\gamma$ is a constant
vector field with $\gamma_n>0$.

By changing the spatial variables one can reduce the boundary condition to the case
\begin{equation}\label{OD1}
D_nu=0 \qquad\mbox{for}\quad x_n=0.
\end{equation}

One of the main results of this paper is the proof of estimate
(\ref{Jan5}) and its analog for the norm
$|\!|\!|\cdot|\!|\!|_{p,q}$, for solutions of the oblique derivative
problem  (\ref{Jan1}), (\ref{OD1}) with arbitrary $p,q\in(1,\infty)$
and for $\mu$ satisfying
\begin{equation*}
-\frac 1p<\mu<1-\frac 1p.
\end{equation*}
In the case of time independent coefficients such estimates for the Neumann problem were
proved in \cite{Na}.

We use an approach based on the study of the Green functions. In Section \ref{Dir} we collect
(partially known) results on the estimate of the Green function of the Dirichlet problem for
equation (\ref{Jan1}). Section \ref{Neu} is devoted to the estimates of the Green function
of problem (\ref{Jan1}), (\ref{OD1}).

In  Section \ref{solv} we apply the obtained estimates to the oblique derivative problem for linear
non-divergence parabolic equations with discontinuous in time coefficients in
cylinders $\Omega\times (0,T)$, where $\Omega$ is a bounded domain
in $\mathbb  R^n$. We prove solvability results in weighted $L_{p,q}$ and $\widetilde
 L_{p,q}$ spaces, where the weight is a power of the distance to the boundary of $\Omega$.
The smoothness of the boundary is characterized by smoothness of
local isomorphisms in neighborhoods of boundary points, which
flatten the boundary. In particular, if the boundary is of the class
${\cal C}^{1,\delta}$ with $\delta\in (0,1]$, then for solutions to
the equation (\ref{Jan1})\footnote{Here the coefficients $a^{ij}$
may depend on $x$ (namely, we assume $a^{ij}\in {\cal C}(\Omega\to L_\infty (0,T))$).}
 in $\Omega\times (0,T)$ with zero initial and boundary conditions
the following coercive estimate is proved in
Theorem \ref{anisotropic} (see Remark \ref{Rem1}):
\begin{eqnarray*}
&&\|(\widehat d(x))^\mu\partial_tu\|_{p,q}+\sum_{ij} \|(\widehat
d(x))^\mu D_iD_ju\|_{p,q}\leq
C\|(\widehat d(x))^\mu f\|_{p,q},\\
&&|\!|\!|(\widehat d(x))^\mu\partial_tu|\!|\!|_{p,q}+\sum_{ij}
|\!|\!|(\widehat d(x))^\mu D_iD_ju|\!|\!|_{p,q}\leq
C|\!|\!|(\widehat d(x))^\mu f|\!|\!|_{p,q},
\end{eqnarray*}
where $\mu$, $p$, $q$ and $\delta$ satisfy $1<p,q<\infty$, \
$1-\delta-\frac{1}{p}<\mu<1-\frac{1}{p}$. \medskip

Let us recall some notation: $x=(x_1,\dots,x_n)=(x',x_n)$ is
a point in $\mathbb R^n$; 
$Du=(D_1u, \dots, D_nu)$ is the gradient of $u$.

We denote
\begin{equation*}
Q_R(x^0,t^0)=\{
(x,t): |x-x^0|<R,\; 0<t^0-t<R^2\};
\end{equation*}
\begin{equation*}
Q_R^+(x^0,t^0)=\{ (x,t): |x-x^0|<R,\; x_n>0,\; 0<t^0-t<R^2\}.
\end{equation*}
The last notation will be used only for $x^0\in\overline{\mathbb R^n_+}$.\medskip

Set
\begin{equation*}
{\cal R}_x=\frac{x_n}{x_n+\sqrt{t-s}},\qquad {\cal
R}_y=\frac{y_n}{y_n+\sqrt{t-s}}.
\end{equation*}

In what follows we denote by the same letter the kernel  and the
corresponding integral operator, i.e.
\begin{equation*}
({\cal K}h)(x,t)=\int\limits_{-\infty}^t\int\limits_{\mathbb R^n}
{\cal K}(x,y;t,s)h(y,s)\,dyds.
 \end{equation*}
Here we expand functions $\cal K$ and $h$ by zero to whole space-time if necessary.\medskip

We adopt the convention regarding summation from $1$ to $n$
with respect to repeated indices.  We use the letter $C$ to denote
various positive constants. To indicate that $C$ depends on some
parameter $a$, we sometimes write $C_a$.

\section{Preliminary results}\label{Dir}
\subsection{The estimates in the whole space and in the half-space
under the Dirichlet boundary condition}\label{Rn}

Let us consider  equation (\ref{Jan1}) in the whole space $\mathbb
R^n$. Using the Fourier transform with respect to  $x$ one can
obtain the following representation of solution through the
right-hand side:
\begin{equation}\label{TTN1}
u(x,t)=\int\limits_{-\infty}^t\int\limits_{{\mathbb R}^n}
\Gamma(x,y;t,s) f(y,s)\ dy\,ds,
\end{equation}
where $\Gamma$ is the Green function of the operator ${\cal L}_0$
 given by
\begin{equation*}
\Gamma(x,y;t,s)= \frac { \det\big(\int_s^t A(\tau)d\tau\big)^{-\frac
12}} {(4\pi)^{\frac n2}}    \exp \bigg(-\frac {\Big(\big(\int_s^t
A(\tau) d\tau\big)^{-1} (x-y),(x-y)\Big)}4\bigg)
\end{equation*}
for $t>s$ and $0$ otherwise.  Here  $A(t)$ is  the matrix $\{
a^{ij}(t)\}_{i,j=1}^n$. The above representation implies, in
particular, the following evident estimates.

\begin{prop}\label{Pr1} Let $\alpha$ and $\beta$ be two arbitrary multi-indices. Then
\begin{equation*}
|D_x^\alpha D_y^\beta \Gamma(x,y;t,s)|\le C\,(t-s)^{-\frac
{n+|\alpha|+|\beta|}2} \,\exp\left(-\frac{\sigma|x-y|^2}{t-s}\right)
\end{equation*}
and
\begin{equation*}
|\partial_sD_x^\alpha D_y^\beta \Gamma(x,y;t,s)|\le C\,(t-s)^{-\frac
{n+|\alpha|+|\beta|}2-1}
\,\exp\left(-\frac{\sigma|x-y|^2}{t-s}\right)
\end{equation*}
for $x,y\in\mathbb R^n$ and $s<t$. Here $\sigma$ depends only on the
ellipticity constant $\nu$ and $C$ may depend on $\nu$, $\alpha$ and $\beta$.
\end{prop}

In the next proposition we present solvability results for equation
(\ref{Jan1}) in the whole space.

\begin{prop}\label{space}
Let $p, q\in\,(1,\infty)$.

{\rm (i)} If $f\in L_{p,q}(\mathbb R^n\times\mathbb R)$, then
the solution of equation {\rm (\ref{Jan1})} given by {\rm (\ref{TTN1})} satisfies
\begin{equation*}
\|\partial_t u\|_{p,q}+\sum_{ij}\| D_iD_ju\|_{p,q}\le C\ \|f\|_{p,q},
\end{equation*}
where $C$ depends only on $\nu$, $p$, $q$.\smallskip

{\rm (ii)} If $f\in \widetilde L_{p,q}(\mathbb R^n\times\mathbb R)$, then
the solution of equation {\rm (\ref{Jan1})} given by {\rm (\ref{TTN1})} satisfies
\begin{equation}\label{est_space2}
|\!|\!|\partial_t u|\!|\!|_{p,q}+\sum_{ij}|\!|\!|
D_iD_ju|\!|\!|_{p,q}\le C\ |\!|\!| f|\!|\!|_{p,q},
\end{equation}
where $C$ depends only on $\nu$, $p$, $q$.
\end{prop}
The first assertion is proved in \cite{Kr2} and the second one in \cite{KN}.\bigskip

We denote  by $\Gamma^{\cal D}(x,y;t,s)$ the Green function of the
operator ${\cal L}_0$ in the half-space ${\mathbb R}^n_+$ subject to
the homogeneous Dirichlet boundary condition on the boundary $x_n=0$.

The next statement is proved in \cite[Theorem 3.6]{KN}.
\begin{prop}\label{T2}
For $x,y\in\mathbb R_+^n$ and $ t>s$ the following estimate is
valid:
\begin{equation}\label{May3}
|D_{x}^{\alpha}D_{y}^{\beta} \Gamma^{\cal D}(x,y;t,s)|\le
C\,\frac{{\cal R}_x^{2-\alpha_n-\varepsilon} {\cal
R}_y^{2-\beta_n-\varepsilon}}{(t-s)^{\frac {n+|\alpha|+|\beta|}2}}
\,\exp \left(-\frac{\sigma|x-y|^2}{t-s}\right),
\end{equation}
where $\sigma$ is a positive number dependent on $\nu$ and $n$,
$\varepsilon$ is an arbitrary small positive number  and $C$ may
depend on $\nu$, $\alpha$, $\beta$ and $\varepsilon$. If
$\alpha_n\le 1$ {\rm(}or $\beta_n\le 1${\rm)} then
$2-\alpha_n-\varepsilon$ {\rm(}$2-\beta_n-\varepsilon${\rm)} must be
replaced by $1-\alpha_n$ {\rm(}$1-\beta_n${\rm)} respectively in the
corresponding exponents.
\end{prop}

Since $(\partial_s+a_{ij}(-s)D_{y_i}D_{y_j})\Gamma^{\cal D}(x,y;t,s)=0$ for $s<t$,
we obtain
\begin{kor}\label{kor1}
For $x,y\in\mathbb R^n_+$ and $t>s$
\begin{equation}\label{May33}
|D_{x}^{\alpha}D_{y}^{\beta}\partial_s \Gamma^{\cal D}(x,y;t,s)|\le
C\,\frac{{\cal R}_x^{2-\alpha_n-\varepsilon}{\cal R}_y^{-\beta_n-\varepsilon}}
{(t-s)^{\frac {n+2+|\alpha|+|\beta|}2}} \,\exp \left(-\frac{\sigma|x-y|^2}{t-s}\right).
\end{equation}
If $\alpha_n\le 1$ then $2-\alpha_n-\varepsilon$ must be replaced by $1-\alpha_n$. 
\end{kor}

\subsection{Coercive estimates for weak solutions to the Dirichlet
problem in the half-space}

We formulate two auxiliary results on estimates of integral operators.
The first statement is a particular case $m=1$ of
\cite[Lemmas A.1 and A.3 and Remark A.2]{KN},
see also \cite[Lemmas 2.1 and 2.2]{Na}.

\begin{prop}\label{L_p}
Let $1\le p\le\infty$, $\sigma>0$, $0<r\le 2$, ${\lambda_1}+{\lambda_2}>-1$,
and let
\begin{equation}\label{mu_m}
-\frac 1p-\lambda_1<\mu<1-\frac 1p+\lambda_2.
\end{equation}
Suppose also that the kernel ${\cal T}(x,y;t,s)$ satisfies the inequality
\begin{equation*}
|{\cal T}(x,y;t,s)|\le C\,\frac {{\cal R}_x^{\lambda_1+r}
{\cal R}_y^{\lambda_2}}{(t-s)^{\frac {n+2-r}2}}\,
\frac{x_n^{\mu-r}}{y_n^{\mu}}\,\exp
\left(-\frac{\sigma|x-y|^2}{t-s} \right),
\end{equation*}
 for $t>s$. Then the integral operator ${\cal T}$ is bounded in
$L_p(\mathbb R^n\times\mathbb R)$ and in $\widetilde L_{p,\infty}(\mathbb R^n\times\mathbb R)$.
\end{prop}

The next proposition is a particular case $m=1$ of \cite[Lemma A.4]{KN}, see also \cite[Lemma 3.2]{Na}.

\begin{prop}\label{L_p_1}
Let $1<p<\infty$, $\sigma>0$, $\varkappa>0$, $0\le r\le 2$,
${\lambda_1}+{\lambda_2}>-1$ and let $\mu$ be subject to {\rm
(\ref{mu_m})}. Also let the kernel ${\cal T}(x,y;t,s)$ satisfy
 the inequality
\begin{equation*}
|{\cal T}(x,y;t,s)|  \le C\,\frac {{\cal R}_x^{\lambda_1+r}
{\cal R}_{y}^{\lambda_2}}{(t-s)^{\frac {n+2-r}2}}
\,\frac{x_n^{\mu-r}}{y_n^{\mu}}\, \left(\frac {\delta}{t-s}\right)^{\varkappa}\!\!\,
\exp \left(-\frac{\sigma|x-y|^2}{t-s} \right),
\end{equation*}
for $t>s+\delta$. Then for any $s^0>0$ the norm of the operator
$${\cal T}\ :\ L_{p,1}(\mathbb R^n\times\ (s^0-\delta,s^0+\delta))\ \to \
L_{p,1}(\mathbb R^n\times\ (s^0+2\delta,\infty))$$ does not exceed
a constant $C$ independent of $\delta$ and $s^0$.
\end{prop}

Now we consider the problem
\begin{equation}\label{2.6}
Lu=f_0+\mbox{div}\,({\bf f})\qquad\mbox{in}\quad \mathbb R^n_+\times\mathbb R
\end{equation}
(here ${\bf f}=(f_1,\ldots,f_n)$) with the boundary condition
\begin{equation}\label{2.7}
u=0\qquad\mbox{for}\quad x_n=0.
\end{equation}

\begin{sats}\label{prop11} Let $1<p,q<\infty$ and $\mu\in (-\frac 1p,1-\frac 1p)$.

\noindent {\rm (i)} Suppose that $x_n^{\mu+1}f_0,\,x_n^\mu {\bf f}\in
\widetilde{L}_{p,q}(\mathbb R^n_+\times\mathbb R)$. Then the function
\begin{equation}\label{2.8}
u(x,t)=\int\limits_{-\infty}^t\int\limits_{\mathbb R^n_+}\Big(\Gamma^{\cal D}(x,y;t,s)f_0(y)
-D_y\Gamma^{\cal D}(x,y;t,s)\cdot{\bf f}(y)\Big)\,dyds
\end{equation}
gives a weak solution of problem {\rm (\ref{2.6})}, {\rm (\ref{2.7})} and satisfies the estimate
\begin{equation}\label{2.9}
|\!|\!|x_n^\mu D u|\!|\!|_{p,q}+|\!|\!|x_n^{\mu-1} u|\!|\!|_{p,q}\leq
C(|\!|\!|x_n^{\mu+1}f_0|\!|\!|_{p,q}+|\!|\!|x_n^\mu {\bf f}|\!|\!|_{p,q}).
\end{equation}

\noindent {\rm (ii)} Suppose that $x_n^{\mu+1}f_0,\,x_n^\mu {\bf f}\in
L_{p,q}(\mathbb R^n_+\times\mathbb R)$. Then the function {\rm (\ref{2.8})}
gives a weak solution of problem {\rm (\ref{2.6})}, {\rm (\ref{2.7})} and satisfies the estimate
\begin{equation}\label{2.9a}
\|x_n^\mu D u\|_{p,q}+\|x_n^{\mu-1} u\|_{p,q}\leq C(\|x_n^{\mu+1}
f_0\|_{p,q}+\|x_n^\mu {\bf f}\|_{p,q}).
\end{equation}
\end{sats}

\begin{proof} First, function (\ref{2.8}) obviously solves
problem (\ref{2.6}), (\ref{2.7}) in the sence of distributions. Thus,
it is sufficient to prove estimates (\ref{2.9}), (\ref{2.9a}).

Put
$$
{\cal K}_0(x,y;t,s)=\frac{x_n^{\mu-1}}{y_n^{\mu+1}}\Gamma^{\cal D}(x,y;t,s);\qquad
{\cal K}_1(x,y;t,s)=\frac{x_n^{\mu-1}}{y_n^\mu}D_{y}\Gamma^{\cal D}(x,y;t,s);
$$
$$
{\cal K}_2(x,y;t,s)=\frac{x_n^{\mu}}{y_n^{\mu+1}} D_x\Gamma_{\cal D}(x,y;t,s);\quad
{\cal K}_3(x,y;t,s)=\frac{x_n^\mu}{y_n^\mu} D_xD_y\Gamma_{\cal D}(x,y;t,s).
$$
\smallskip

(i) By Proposition \ref{T2} the kernels ${\cal K}_0$ and ${\cal K}_1$ satisfy the conditions of
Proposition \ref{L_p} with $r=1$ and

with $\lambda_1=-1$, $\lambda_2=1$ and $\mu$ replaced by $\mu+1$ for the kernel ${\cal K}_0$;

with $\lambda_1=\lambda_2=0$ for the kernel ${\cal K}_1$, respectively.\smallskip

\noindent This implies that for $\mu\in (-\frac 1p,1-\frac 1p)$
\begin{equation}\label{2.5a}
\|x_n^{\mu-1} u\|_p\leq C(\|x_n^{\mu+1} f_0\|_p+\|x_n^\mu {\bf f}\|_p)
\end{equation}
and
\begin{equation}\label{2.5b}
|\!|\!|x_n^{\mu-1} u|\!|\!|_{p,\infty}\leq
C(|\!|\!|x_n^{\mu+1}f_0|\!|\!|_{p,\infty}+|\!|\!|x_n^\mu {\bf f}|\!|\!|_{p,\infty}).
\end{equation}

Interpolating (\ref{2.5a}) and (\ref{2.5b}) we arrive at
\begin{equation}\label{2.5azz}
|\!|\!|x_n^{\mu-1}u|\!|\!|_{p,q}\leq C(|\!|\!|x_n^{\mu+1}
f_0|\!|\!|_{p,q}+|\!|\!|x_n^\mu {\bf f}|\!|\!|_{p,q}),
\end{equation}
for $1<p\leq q<\infty$ and $\mu\in (-\frac 1p,1-\frac 1p)$. Now duality
argument gives (\ref{2.5azz}) for all $1<p,q<\infty$ and for the
same interval for $\mu$.

To estimate the first term in the left-hand side of (\ref{2.9}) we use local estimates. We put
\begin{equation*}
B_{\rho,\vartheta}(\xi)=\{ x\in\mathbb R^n\,:\,
|x'-\xi'|<\rho,\frac{\rho}{\vartheta}<x_n<\rho\}.
\end{equation*}
Localization of estimate (\ref{est_space2}) using an appropriate cut-off function,
which is equal to $1$ on $B_{\rho,2}$ and $0$ outside $B_{2\rho,8}$, gives
\begin{equation*}\label{2.11}
\int\limits_{B_{\rho,2}(\xi)}\Big (\int\limits_{\mathbb R}|Du|^qdt\Big)^{\frac pq}dx\leq
C\int\limits_{B_{2\rho,8}(\xi)}\Big (\int\limits_{\mathbb R}(|u|^q\rho^{-q}+
\rho^q|f_0|^q+|{\bf f}|^qdt\Big )^{\frac pq}dx.
\end{equation*}
Using a proper partition of unity in ${\mathbb R^n_+}$, we arrive at
\begin{multline*}
\int\limits_{\mathbb R^n_+}\bigg(\int\limits_{\mathbb R} |Du|^q dt\bigg)^{p/q}x_n^{\mu p} dx
\le C\bigg(\ \int\limits_{\mathbb R^n_+}\bigg(\int\limits_{\mathbb R}|u|^q\, dt\bigg)^{p/q}
x_n^{\mu p-p} dx \\
+ \int\limits_{\mathbb R^n_+}\bigg(\int\limits_{\mathbb R}|{\bf f}|^q\, dt\bigg)^{p/q}x_n^{\mu p} dx
+ \int\limits_{\mathbb R^n_+}\bigg(\int\limits_{\mathbb R}|f_0|^q\, dt\bigg)^{p/q}
x_n^{\mu p+p} dx\bigg).
\end{multline*}
This immediately implies (\ref{2.9}) with regard of (\ref{2.5azz}).\medskip


(ii) To deal with the scale $L_{p,q}$, we need the following lemma.

\begin{lem}\label{weak_2}
Let a function $h$ be supported in the layer $|s-s^0|\le\delta$ and
satisfy $\int h(y;s)\ ds\equiv 0$. Also let $p\in (1,\infty)$ and
$\mu\in (-\frac 1p,1-\frac 1p)$.
Then the operators ${\cal K}_j$, $j=0,1,2,3$, satisfy
$$
\int\limits_{|t- s^0|>2\delta}\Vert ({\cal K}_jh)(\cdot; t)\Vert_p\ dt\le C\,\Vert h\Vert_{p,1},
$$
where $C$ does not depend on $\delta$ and $ s^0$.
\end{lem}

\begin{proof}
By $\int h(y; s)\, ds\equiv 0$, we have
\begin{equation}\label{difference1}
({\cal K}_jh)(x;t)=\int\limits_{-\infty}^{t}\int\limits_{\mathbb R^n}
\Bigl({\cal K}_j(x,y;t, s)-{\cal K}_j(x,y;t,s^0)\Bigr)\, h(y; s)\, dy\, ds
\end{equation}
(we recall that all functions are assumed to be extended by zero).

We choose $\varepsilon>0$ such that
\begin{equation}\label{mueps}
-\frac 1p<\mu<1-\frac 1p-\varepsilon.
\end{equation}
For $| s- s^0|<\delta$ and $t- s^0>2\delta$, estimate (\ref{May33})
implies
\begin{eqnarray*}
&&\left|{\cal K}_j(x,y;t, s)-{\cal K}_j(x,y;t, s^0)\right|
\le\int\limits_{s^0}^s|\partial_\tau {\cal K}_j(x,y;t,\tau)|\,d\tau\\
&&\le C\,\frac {{\cal R}^{\ell_1}_{x} {\cal R}^{\ell_2-\varepsilon}_{y}}
{(t- s)^{\frac {n+2-r}2}}\, \frac {x_n^{\ell_3}}{y_n^{\ell_4}}\,
 \frac {\delta} {t- s}\, \exp \left(-\frac {\sigma|x-y|^2}{t- s}\right),
\end{eqnarray*}

with $r=2$, $\ell_1=1$, $\ell_2=0$, $\ell_3=\mu-1$, $\ell_4=\mu+1$ for the kernel ${\cal K}_0$;

with $r=1$, $\ell_1=1$, $\ell_2=-1$, $\ell_3=\mu-1$, $\ell_4=\mu$ for the kernel ${\cal K}_1$;

with $r=1$, $\ell_1=0$, $\ell_2=0$, $\ell_3=\mu$, $\ell_4=\mu+1$ for the kernel ${\cal K}_2$;

with $r=0$, $\ell_1=0$, $\ell_2=-1$, $\ell_3=\mu$, $\ell_4=\mu$ for the kernel ${\cal K}_3$.\smallskip

\noindent On the other hand, estimate (\ref{May3}) implies
\begin{equation*}
\left|{\cal K}_j(x,y;t, s)-{\cal K}_j(x,y;t, s^0)\right| \le C\,\frac
{{\cal R}^{\ell_1}_{x}{\cal R}^{\ell_2+1}_{y}}{(t-s)^{\frac {n+2-r}2}}
\, \frac {x_n^{\ell_3}}{y_n^{\ell_4}}\, \exp\left(-\frac{\sigma|x-y|^2}{t-s} \right).
\end{equation*}
Combination of these estimates gives
\begin{equation*}
\left|{\cal K}_j(x,y;t, s)-{\cal K}_j(x,y;t, s^0)\right|
\le \frac {C\delta^\varkappa\,{\cal R}^{\ell_1}_{x} {\cal R}^{\ell_2+1-\varepsilon}_{y}}
{(t- s)^{\frac {n+2-r}2+\varkappa}}\, \frac {x_n^{\ell_3}}{y_n^{\ell_4}}\,
\exp \left(-\frac {\sigma|x-y|^2}{t- s}\right),
\end{equation*}
where $\varkappa=\frac {\varepsilon}{1+\varepsilon}$.
Thus, the kernels in (\ref{difference1}) satisfy the assumptions of Proposition \ref{L_p_1}

with $\lambda_1=-1$, $\lambda_2=1-\varepsilon$ and $\mu$ replaced by $\mu+1$ for kernels
${\cal K}_0$ and ${\cal K}_2$;

with $\lambda_1=0$, $\lambda_2=-\varepsilon$ for kernels ${\cal K}_1$ and ${\cal K}_3$, respectively.\smallskip

\noindent Inequality (\ref{mueps}) becomes (\ref{mu_m}), and the Lemma follows.\end{proof}

\medskip

We continue the proof of the second statement of Theorem \ref{prop11}.
Estimate (\ref{2.9}) for $q=p$ provides boundedness of the
operators ${\cal K}_j$, $j=0,1,2,3$, in $L_p(\mathbb R^n\times \mathbb R)$, 
which gives the first condition in \cite[Theorem 3.8]{BIN}. Lemma \ref{weak_2} is equivalent to the
second condition in this theorem. Therefore, Theorem 3.8 \cite{BIN}
ensures that these operators are bounded in $L_{p,q}(\mathbb R^n\times \mathbb R)$ for any
$q\in\,(1,p)$. For $q\in\,(p,\infty)$ this statement follows by duality arguments.
This implies estimate (\ref{2.9a}).\end{proof}

\section{Oblique derivative problem}\label{Neu}

\subsection{The Green function}

\begin{sats} \label{Th2}
There exists a Green function 
$\Gamma^{\cal N}=\Gamma^{\cal N}(x,y;t,s)$ of problem
{\rm (\ref{Jan1})}, {\rm (\ref{OD1})} and for arbitrary
$x,y\in\mathbb R_+^n$ and $ t>s$ it satisfies the estimate
\begin{equation}\label{Ap1a}
|D^\alpha_xD^\beta_y\Gamma^{\cal N}(x,y;t,s)|\leq C\frac{{\cal
R}_x^{\widehat{\alpha}_n}{\cal
R}_y^{\widehat{\beta}_n}}{(t-s)^{\frac {n+|\alpha|+|\beta|}2}}\,\exp
\left(-\frac{\sigma|x-y|^2}{t-s}\right),
\end{equation}
\begin{equation}\label{May34}
|D_{x}^{\alpha}D_{y}^{\beta}\partial_s \Gamma^{\cal N}(x,y;t,s)|\le
C\,\frac{{\cal R}_x^{\widehat\alpha_n}{\cal R}_y^{-1-\beta_n-\varepsilon}}
{(t-s)^{\frac {n+2+|\alpha|+|\beta|}2}} \,\exp \left(-\frac{\sigma|x-y|^2}{t-s}\right),
\end{equation}
where 
$$
\widehat{\alpha}_n=\begin{cases}
0 , & \alpha_n=0;\\ 
2-\alpha_n, & \alpha_n=1,2;\\
3-\alpha_n-\varepsilon, & \alpha_n\geq 3;
\end{cases}
\qquad \qquad 
\widehat{\beta}_n=\begin{cases}
0, & \beta_n=0;\\
1-\beta_n-\varepsilon, & \beta_n\geq 1.
\end{cases}
$$
Here $\sigma$ is a positive number dependent on $\nu$ and $n$,
$\varepsilon$ is an arbitrary small positive number  and $C$ may
depend on $\nu$, $\alpha$, $\beta$ and $\varepsilon$.
\end{sats}

\begin{proof} Let $u$ be a solution of problem {\rm (\ref{Jan1})}, {\rm (\ref{OD1})}. 
Then the derivative $D_nu$ obviously satisfies the Dirichlet problem
(\ref{2.6}), (\ref{2.7}) with $f_0=0$ and ${\bf f}=(0,\ldots,0,f)$.
Therefore, 
$$
D_nu=-\int\limits_{-\infty}^t\int\limits_{\mathbb R^n_+}
D_{y_n}\Gamma^{\cal D}(x,y;t,s)f(y;s)\,dyds,
$$
and we can write solution to problem (\ref{Jan1}), (\ref{OD1}) as
\begin{equation}\label{2.11ag}
u(x;t)=\int\limits_{-\infty}^t\int\limits_{\mathbb R^n_+}\Gamma^{\cal N}(x,y,t,s)f(y,s)\,dyds,
\end{equation}
where
$$
\Gamma^{\cal N}(x,y;t,s)=\int\limits_{x_n}^\infty D_{y_n}
\Gamma^{\cal D}(x',z_n,y;t,s)\,dz_n.
$$

Since $D_{x_n}\Gamma^{\cal N}(x,y;t,s)=-D_{y_n}\Gamma^{\cal D}(x,y;t,s)$, 
we derive from (\ref{May3}) that
\begin{equation}\label{March29b}
|D_{x}^{\alpha}D_{y}^{\beta} D_{x_n}\Gamma^{\cal N}(x,y;t,s)|\le
C\,\frac{{\cal R}_x^{2-\alpha_n-\varepsilon} {\cal
R}_y^{1-\beta_n-\varepsilon}}{(t-s)^{\frac {n+1+|\alpha|+|\beta|}2}}
\,\exp \left(-\frac{\sigma|x-y|^2}{t-s}\right),
\end{equation}
where $2-\alpha_n-\varepsilon$ must be replaced by $1-\alpha_n$ if
$\alpha_n\leq 1$ and $1-\beta_n-\varepsilon$ by $0$ if
$\beta_n=0$. Estimate (\ref{Ap1a}) with $\alpha_n\geq 1$ follows
from (\ref{March29b}). 

In a similar way we derive from (\ref{May33}) that
\begin{eqnarray}\label{May9}
&&|D_{x}^{\alpha} D_{x_n}D_{y}^{\beta}\partial_s\Gamma^{\cal N}(x,y;t,s)|\nonumber\\
&&\le C\,\frac{{\cal R}_x^{2-\alpha_n-\varepsilon}{\cal R}_y^{-1-\beta_n-\varepsilon}}
{(t-s)^{\frac {n+3+|\alpha|+|\beta|}2}} \,\exp \left(-\frac{\sigma|x-y|^2}{t-s}\right),
\end{eqnarray}
where $2-\alpha_n-\varepsilon$ must be replaced by $1-\alpha_n$ if
$\alpha_n\leq 1$. Estimate (\ref{May34}) with $\alpha_n\geq 1$ follows
from (\ref{May9}).\medskip

To estimate derivatives with respect to $x'$ we consider two cases.\medskip

{\it Case 1}: $|x_n-y_n|\le\sqrt{t-s}$. Then (\ref{March29b}) implies
\begin{eqnarray*}
&&|D_{x'}^{\alpha'}D_{y}^{\beta}\Gamma^{\cal N}(x,y;t,s)|\le\int\limits_{x_n}^\infty 
|D_{x'}^{\alpha'}D_{y}^{\beta} D_{z_n}\Gamma^{\cal N}(x',z_n,y;t,s)|\,dz_n\\
&&\le \frac{C\,{\cal R}_y^{\widehat\beta_n}}{(t-s)^{\frac {n+|\alpha'|+|\beta|}2}}
\,\exp \left(-\frac{\sigma|x'-y'|^2}{t-s}\right) \int\limits_{\mathbb R}
\exp \left(-\frac{\sigma|z_n-y_n|^2}{t-s}\right)\,\frac {dz_n}{\sqrt{t-s}}\\
&&\le \frac{C\,{\cal R}_y^{\widehat\beta_n}}{(t-s)^{\frac {n+|\alpha'|+|\beta|}2}}
\,\exp \left(-\frac{\sigma|x-y|^2}{t-s}+\sigma\right)
\end{eqnarray*}
(the last inequality is due to $\frac {|x_n-y_n|^2}{t-s}\le1$), which gives (\ref{Ap1a}) with 
$\alpha_n=0$ in the case 1. In a similar way we derive estimate (\ref{May34}) 
with $\alpha_n=0$ in the case 1 from (\ref{May9}).
\medskip

{\it Case 2}: $|x_n-y_n|>\sqrt{t-s}$. Then we rewrite equation ${\cal L}_0\Gamma^{\cal N}=0$ as
\begin{eqnarray}\label{March29c}
&{\cal L}'_0\Gamma^{\cal N}&\equiv \partial_t\Gamma^{\cal N}-
\sum_{i,j=1}^{n-1}a_{ij}(t)D_{x_i}D_{x_j}\Gamma^{\cal N}
\nonumber\\
&&={\cal F}
\equiv\Big(2\sum_{j=1}^{n-1}a_{jn}D_{x_j}D_{x_n}\Gamma^{\cal N}
+a_{nn}D_{x_n}D_{x_n}\Gamma^{\cal N}\Big )
.
\end{eqnarray}
From (\ref{March29b}) and (\ref{March29c}) it follows that
\begin{equation}\label{March30c}
|D_{x'}^{\alpha'}D_y^\beta {\cal F}(x,y;t,s)|\leq C\,\frac{{\cal R}_y^{\widehat\beta_n}}
{(t-s)^{\frac {n+2+|\alpha'|+|\beta|}2}} \,\exp \left(-\frac{\sigma|x-y|^2}{t-s}\right).
\end{equation}

Let $\Gamma'(x',y';t,s)$ be the Green function of
the operator ${\cal L}'_0$ in $\mathbb R^{n-1}\times\mathbb R$. Then solving
(\ref{March29c}), we get
\begin{equation*}
\Gamma^{\cal N}(x,y;t,s)=\int\limits_{s}^t\int\limits_{\mathbb R^{n-1}}
\Gamma'(x',z';t,\tau){\cal F}(z',x_n,y;\tau,s)\,dz'd\tau.
\end{equation*}
Since $\Gamma'(x',z';t,\tau)$ depends only on the difference $x'-z'$, we obtain
\begin{equation}\label{March29d}
D_{x'}^{\alpha'}\Gamma^{\cal N}(x,y;t,s)=\int\limits_{s}^t\int\limits_{\mathbb R^{n-1}}
\Gamma'(x',z';t,\tau)D_{z'}^{\alpha'}{\cal F}(z',x_n,y;\tau,s)\,dz'd\tau.
\end{equation}

Using Proposition \ref{Pr1} for $\Gamma'$ we get from
(\ref{March29d}) and (\ref{March30c}) 
\begin{eqnarray*}\label{QQ1}
&&|D_{x'}^{\alpha'}D_y^\beta\Gamma^{\cal N}(x,y;t,s)|\leq
\int\limits_{s}^t\int\limits_{\mathbb R^{n-1}}\frac{C}{(t-\tau)^{\frac {n-1}2}}
 \,\exp \left(-\frac{\sigma|x'-z'|^2}{t-\tau}\right)\nonumber\\
&&\times\frac{{\cal R}_y^{\widehat\beta_n}}{(\tau-s)^{\frac {n+2+|\alpha'|+|\beta|}2}} 
\,\exp \left(-\frac{\sigma|(z',x_n)-y|^2}{\tau-s}\right)\,dz'd\tau.
\end{eqnarray*}
We observe that ${\cal R}_y$ here has non-standard time argument: $\tau-s$ instead of
$t-s$. However, since $\widehat\beta_n\le0$, we can estimate ``non-standard''
${\cal R}_y^{\widehat\beta_n}$ by standard one.

Integrating with respect to $z'$, herewith using Fourier transform, we get
\begin{eqnarray*}
&&|D_{x'}^{\alpha'}D_y^\beta\Gamma^{\cal N}(x,y;t,s)|\leq
\frac{C\,{\cal R}_y^{\widehat\beta_n}}{(t-s)^{\frac {n-1}2}}\exp
\left(-\frac{\sigma|x'-y'|^2}{t-s}\right)\nonumber\\
&&\times\int\limits_{s}^t
\frac{1}{(\tau-s)^{\frac {3+|\alpha'|+|\beta|}2}} \,\exp
\left(-\frac{\sigma (x_n-y_n)^2}{\tau-s}\right)\,d\tau.
\end{eqnarray*}
Substituting $\theta=\frac {t-\tau}{\tau-s}$, we arrive at
\begin{eqnarray*}
&&|D_{x'}^{\alpha'}D_y^\beta\Gamma^{\cal N}(x,y;t,s)|\leq
\frac{C\,{\cal R}_y^{\widehat\beta_n}}{(t-s)^{\frac {n+|\alpha'|+|\beta|}2}}
\,\exp \left(-\frac{\sigma|x'-y'|^2}{t-s}\right)\\
&&\times\int\limits_0^\infty (\theta+1)^{\frac {|\alpha'|+|\beta|-1}2} 
\,\exp \left(-\frac{\sigma (x_n-y_n)^2}{t-s}\,(\theta+1)\right)\,d\theta.
\end{eqnarray*}
Since $\frac {|x_n-y_n|^2}{t-s}>1$, this implies
\begin{eqnarray*}
&|D_{x'}^{\alpha'}D_y^\beta\Gamma^{\cal N}(x,y;t,\tau)|&\leq
\frac{C\,{\cal R}_y^{\widehat\beta_n}}{(t-s)^{\frac {n+|\alpha'|+|\beta|}2}}
\,\exp\left(-\frac{\sigma|x-y|^2}{t-s}\right)\\
&&\times\int\limits_0^\infty (\theta+1)^{\frac {|\alpha'|+|\beta|-1}2} 
\,\exp \left(-\sigma\theta\right)\,d\theta,
\end{eqnarray*}
which gives (\ref{Ap1a}) with $\alpha_n=0$ in the case 2.

 In a similar way we derive the estimate (\ref{May34}) 
with $\alpha_n=0$ in the case 2, and the proof is complete.
\end{proof}



\subsection{Coercive estimates in $\widetilde{L}_{p,q}$ and in $L_{p,q}$}\label{Sect2.2}

\begin{sats}\label{Th1s} Let $1<p,q<\infty$ and $\mu\in (-\frac 1p,1-\frac 1p)$.\smallskip

(i) If $f\in\widetilde{L}_{p,q}(\mathbb R^n_+\times\mathbb R)$ then solution {\rm (\ref{2.11ag})} 
to problem {\rm (\ref{Jan1})}, {\rm (\ref{OD1})} satisfies
\begin{equation}\label{TTn1a}
 |\!|\!|x_n^\mu\partial_t u|\!|\!|_{p,q}+|\!|\!|x_n^\mu D(Du)|\!|\!|_{p,q}\le 
C\ |\!|\!|x_n^\mu f|\!|\!|_{p,q}.
\end{equation}

(ii) If $f\in L_{p,q}(\mathbb R^n_+\times\mathbb R)$ then solution {\rm (\ref{2.11ag})} 
to problem {\rm (\ref{Jan1})}, {\rm (\ref{OD1})} satisfies
\begin{equation}\label{est_halfspace}
\|x_n^\mu\partial_t u\|_{p,q}+\|x_n^\mu D(Du)\|_{p,q}\le C\ \|x_n^\mu f\|_{p,q}.
\end{equation}
The constant $C$ 
depends only on $\nu$, $\mu$, $p$ and $q$.
\end{sats}

\begin{proof}
First, we recall that the function $D_nu$ satisfies the Dirichlet problem (\ref{2.6}), (\ref{2.7}) 
with $f_0=0$ and ${\bf f}=(0,\ldots,0,f)$. Thus, Theorem \ref{prop11} gives
\begin{eqnarray}\label{2.11a}
|\!|\!|x_n^\mu D(D_nu)|\!|\!|_{p,q}\leq C|\!|\!|x_n^\mu f|\!|\!|_{p,q};\\
\|x_n^\mu D(D_nu)\|_{p,q}\leq C\|x_n^\mu f\|_{p,q}.
\label{2.11b}\end{eqnarray}

To estimate the derivatives $D'D'u$ in $\widetilde L_{p,q}$-norm, we proceed similarly to Theorem 
\ref{Th2}. We rewrite equation (\ref{Jan1}) as in (\ref{March29c}):
\begin{equation*}
{\cal L}'_0u=\widetilde{f}\equiv f+2\sum_{j=1}^{n-1}a_{jn}D_jD_nu+a_{nn}D_nD_nu.
\end{equation*}

Using Proposition \ref{space} (ii) in $\mathbb R^{n-1}$ we obtain
\begin{equation}\label{3.2} |\!|\!|D'D'u(\cdot,x_n)|\!|\!|_{p,q}\leq
C|\!|\!|\widetilde{f}(\cdot,x_n)|\!|\!|_{p,q}
\end{equation}
almost for all $x_n>0$. Multiplying both sides of (\ref{3.2}) by
$x_n^\mu$ and taking $L_p$ norm with respect to $x_n$, we arrive at
\begin{equation}\label{3.3}
|\!|\!|x_n^\mu D'D'u|\!|\!|_{p,q}\leq C|\!|\!|x_n^\mu
\widetilde{f}|\!|\!|_{p,q}\leq C|\!|\!|x_n^\mu f|\!|\!|_{p,q},
\end{equation}
where we have used estimate (\ref{2.11a}). The first term in (\ref{TTn1a}) is estimated by using
(\ref{2.11a}), (\ref{3.3}) and equation (\ref{Jan1}), and the statement (i) follows.\medskip

For $L_{p,q}$-norm of $D'D'u$ this approach fails, so we proceed as in the part (ii)
of Theorem \ref{prop11}. Let us introduce the kernels
$$
{\cal K}_4(x,y;t,s)=\frac{x_n^\mu}{y_n^\mu}D'_{x'}D'_{x'}{\Gamma}_{\cal N}(x,y;t,s);\qquad
{\cal K}^*_4(x,y;t,s)={\cal K}_4(y,x;s,t).
$$
Estimate (\ref{3.3}) with $q=p$ means that the operator ${\cal K}_4$ is bounded in
$L_p(\mathbb R^n\times \mathbb R)$. 
Choose $\varepsilon>0$ such that relation (\ref{mueps}) holds. Using
estimates (\ref{Ap1a}) and (\ref{May34}), it is easy to check that
${\cal K}_4$ satisfies the same estimates as the kernel ${\cal K}_3$
in Theorem \ref{prop11}. Verbatim repetition of arguments shows that
this operator is bounded in $L_{p,q}(\mathbb R^n\times \mathbb R)$ for any
$q\in\,(1,p)$. 

Further, by duality the operator ${\cal K}^*_4$ is bounded in 
$L_{p'}(\mathbb R^n\times \mathbb R)$. 
Using (\ref{Ap1a}) and relation 
$(\partial_s-a_{ij}(s)D_{y_i}D_{y_j})\Gamma^{\cal N}(y,x;s,t)=0$ for $s>t$,
we obtain 
$$
|\partial_s{\cal K}^*_4(x,y;t,s)|\le \frac C{(s-t)^{\frac {n+4}2}}\, 
\frac {x_n^{-\mu}}{y_n^{-\mu}}\,
\exp \left(-\frac {\sigma|x-y|^2}{t- s}\right).
$$
For $| s- s^0|<\delta$ and $s^0-t>2\delta$ this implies 
\begin{equation*}
\left|{\cal K}^*_4(x,y;t, s)-{\cal K}^*_4(x,y;t, s^0)\right|
\le \frac {C\delta}
{(t- s)^{\frac {n+2}2+1}}\, \frac {x_n^{-\mu}}{y_n^{-\mu}}\,
\exp \left(-\frac {\sigma|x-y|^2}{t- s}\right).
\end{equation*}
The last estimate allows us to apply Proposition \ref{L_p_1}
with $\varkappa=1$, $r=0$, $\lambda_1=\lambda_2=0$ and $p$ replaced by $p'$.
Therefore, Theorem 3.8 \cite{BIN} ensures that for any
$q\in\,(p,\infty)$ the operator ${\cal K}^*_4$ is bounded in 
$L_{p',q'}(\mathbb R^n\times \mathbb R)$. By duality
 the operator ${\cal K}_4$ is bounded in 
$L_{p,q}(\mathbb R^n\times \mathbb R)$.

Thus, we have
\begin{equation}\label{3.4}
\|x_n^\mu D'D'u\|_{p,q}\leq C\|x_n^\mu f\|_{p,q}
\end{equation}
for all $1<q<\infty$. The first term in (\ref{est_halfspace}) is estimated by (\ref{2.11b}), 
(\ref{3.4}) and equation (\ref{Jan1}), and the statement (ii) also follows.
\end{proof}

\section{Solvability of the oblique derivative problem in a bounded cylinder}\label{solv}

Let $\Omega$ be a bounded domain in $\mathbb R^n$ with boundary
$\partial\Omega$. For a cylinder $Q=\Omega\times(0,T)$, we denote by
$\partial''Q=\partial\varOmega\times(0,T)$ its lateral boundary.

We introduce two scales of functional spaces: ${\mathbb
L}_{p,q,(\mu)}(Q)$ and $\widetilde{\mathbb L}_{p,q,(\mu)}(Q)$, with
norms
\begin{equation*}
{\pmb \|}f{\pmb \|}_{p,q,(\mu),Q}= \Vert (\widehat
d(x))^{\mu}f\Vert_{p,q,Q}= \Big (\int\limits_0^T\Big
(\int\limits_\Omega (\widehat d(x))^{\mu p}|f(x,t)|^pdx\Big
)^{\frac qp}dt\Big )^{\frac 1q}
\end{equation*}
and
\begin{equation*}
\pmb {|\!|\!|}f\pmb {|\!|\!|}_{p,q,(\mu),Q}= |\!|\!| (\widehat
d(x))^{\mu}f|\!|\!|_{p,q,Q}= \Big (\int\limits_\Omega\Big
(\int\limits_0^T (\widehat d(x))^{\mu q}|f(x,t)|^qdt\Big
)^{\frac pq}dx\Big)^{\frac 1p}
\end{equation*}
respectively, where $\widehat d(x)$ stands for the distance from $x
\in \Omega$ to $\partial\Omega$. For $p=q$ these spaces coincide,
and we use the notation ${\mathbb L}_{p,(\mu)}(Q)$ and ${\pmb \|}\cdot{\pmb \|}_{p,(\mu),Q}$.

We denote by ${\mathbb W}^{2,1}_{p,q,(\mu)} (Q)$ and $\widetilde
{\mathbb W}^{2,1}_{p,q,(\mu)} (Q)$ the set of functions with the
finite seminorms
\begin{equation*}
{\pmb \|}\partial_tu {\pmb \|}_{p,q,(\mu),Q}+\sum_{ij}{\pmb \|}
D_iD_ju {\pmb \|} _{p,q, (\mu),Q}
\end{equation*}
and
\begin{equation*}
\pmb{|\!|\!|} \partial_tu \pmb{|\!|\!|}_{p,q, (\mu),Q}+\sum_{ij}
\pmb{|\!|\!|} D_iD_ju \pmb{|\!|\!|}_{p,q, (\mu),Q}
\end{equation*}
respectively. These seminorms become norms on the subspaces defined
by $u|_{t=0}=0$. For $p=q$ we write ${\mathbb W}^{2,1}_{p,(\mu)} (Q)$.\medskip

We say $\partial\Omega \in {\cal W}^2_{p,(\mu)}$ if for any point
$x^0\in\partial\Omega$ there exists a neighborhood $\cal U$  and a
diffeomorphism $\Psi$ mapping ${\cal U}\cap\Omega$ onto the half-ball 
$B_1^+$ and satisfying 
\begin{equation*}
(\widehat d(x))^{\mu}D^2\Psi\in L_p({\cal U}\cap\Omega);
\qquad x_n^\mu D^2\Psi^{-1}\in L_p(B_1^+),
\end{equation*} 
where corresponding norms are uniformly bounded with respect to $x^0$.\smallskip

We set $\widehat\mu(p,q)=1-\frac{n}{p}-\frac{2}{q}$.\medskip


We consider the initial-boundary value problem
\begin{eqnarray}\label{DesS4}
&&{\cal L}u\equiv \partial_tu-a^{ij}(x,t)D_iD_ju+b^i(x,t)D_iu=f(x,t) \quad
{\textup{in}} \quad Q;\\
&&\gamma^i(x,t)D_iu|_{\partial''Q}=0,\qquad\qquad u|_{t=0}=0.\nonumber
\end{eqnarray}
The matrix of leading coefficients $a^{ij}\in {\cal C}(\overline{\Omega}\to L_\infty(0,T))$ 
is symmetric and satisfies the ellipticity condition (\ref{Jan2}). The vector field $\gamma$ is 
assumed non-tangent to $\partial\Omega$:
\begin{equation}\label{nontang}
\gamma^i(x,t){\bf n}_i(x)\ge\gamma_0,\qquad
(x,t)\in\partial''Q, \quad \gamma_0=\const>0
\end{equation}
(here ${\bf n}(x)$ stands for the unit exterior normal vector to $\partial\Omega$ at the 
point $x$).

\begin{sats}\label{anisotropic}
Let $1<p,q<\infty$ and $\mu \in (-\frac{1}{p},1-\frac{1}{p})$. Assume that the components $\gamma^i$ belong to the anisotropic 
H\"older space ${\cal C}^{0,1;\frac 12}(\partial''Q)$. \medskip

{\bf 1}. Let $ b^i \in {\mathbb
L}_{{\overline{p}},{\overline{q}},(\overline{\mu})}(Q)+ {\mathbb
L}_{\infty,(\overline{\overline{\mu}})}(Q)$, where $\displaystyle
{\overline{p}}$ and $\displaystyle {\overline{q}}$ are subject to
\begin{equation*}\overline{p}\ge p;\quad \left[
\begin{array}{ll}
\overline{q}=q;&
\widehat\mu(\overline{p},\overline{q})>0\\
q<\overline{q}<\infty;& \widehat\mu(\overline{p},\overline{q})=0
\end{array}\right.,
\end{equation*}
while $\overline {\mu}$ and $\overline{\overline{\mu}}$ satisfy
\begin{equation}\label{mumu}
{\overline{\mu}}=\min\{\mu, \max\{\widehat\mu(p,q),0\}\};\qquad
\overline{\overline{\mu}}<\mu+\textstyle\frac 1p.
\end{equation}

Suppose also that either $\partial\Omega \in {\cal W}^2_{\infty,(\overline{\overline{\mu}})}$ 
or $\partial\Omega \in {\cal W}^2_{\overline{p},(\overline{\mu})}$. Then, for any 
$f \in {\mathbb L}_{p,q,(\mu)}(Q)$, the initial-boundary value problem
{\rm (\ref{DesS4})} has a unique solution $u\in {\mathbb
W}^{2,1}_{p,q,(\mu)}(Q)$. Moreover, this solution satisfies
\begin{equation*}
{\pmb \|}\partial_tu{\pmb \|}_{p,q,(\mu)}+
\sum_{ij}{\pmb \|}D_iD_ju{\pmb \|}_{p,q,(\mu)} \le C {\pmb
\|}f{\pmb \|}_{p,q,(\mu)},
\end{equation*}
where the positive constant $C$ does not depend on $f$.\medskip

{\bf 2}. Let $ b^i \in \widetilde{\mathbb L}_{{\overline{p}},{\overline{q}},(\overline{\mu})}(Q)+ 
{\mathbb L}_{\infty,(\overline{\overline{\mu}})}(Q)$, where $\displaystyle {\overline{p}}$ and 
$\displaystyle {\overline{q}}$ are subject to
\begin{equation*}\overline{q}\ge q;\quad \left[
\begin{array}{ll}
\overline{p}=p;&
\widehat\mu(\overline{p},\overline{q})>0\\
p<\overline{p}<\infty;& \widehat\mu(\overline{p},\overline{q})=0
\end{array}\right.,
\end{equation*}
while $\overline {\mu}$ and $\overline{\overline{\mu}}$ satisfy
(\ref{mumu}). Suppose also that $\partial\Omega$ satisfies the
same conditions as in the part {\bf 1}. Then, for any $f \in
\widetilde{\mathbb L}_{p,q,(\mu)}(Q)$, the 
problem {\rm (\ref{DesS4})} has a unique solution
$u\in\widetilde{\mathbb W}^{2,1}_{p,q,(\mu)}(Q)$. Moreover, this
solution satisfies
\begin{equation*}
\pmb{|\!|\!|} \partial_tu
\pmb{|\!|\!|}_{p,q,(\mu)}+ \sum_{ij}\pmb{|\!|\!|} D_iD_ju
\pmb{|\!|\!|}_{p,q,(\mu)} \le C \pmb{|\!|\!|} f
\pmb{|\!|\!|}_{p,q,(\mu)},
\end{equation*}
where the positive constant $C$ does not depend on $f$.

\end{sats}

\begin{Rem}\label{Rem1}
 It is well known (see, e.g., \cite{Li}
)
that if $\partial\Omega \in {\cal C}^{1,\delta}$ for some $\delta\in(0,1]$, then 
$\partial\Omega \in {\cal W}^2_{\infty,(1-\delta)}$. 
In this case the second inequality in {\rm (\ref{mumu})} implies solvability 
of the problem {\rm (\ref{DesS4})} for $1-\delta-\frac{1}{p}<\mu<1-\frac{1}{p}$.
\end{Rem}

\begin{proof}
The standard scheme, see \cite[Ch.IV, \S9]{LSU}, including partition
of unity, local rectifying of $\partial\Omega$ and coefficients
freezing, reduces the proof to the coercive estimates for the model
problems to equation (\ref{Jan1}) in the whole space and in the
half-space. These estimates are obtained in \cite[Theorem 1.1]{Kr}
and our Theorem \ref{Th1s}. By the H\"older inequality and the
embedding theorems (see, e.g., \cite[Theorems 10.1 and 10.4]{BIN}),
the assumptions on $b^i$ guarantee that the lower-order terms in
(\ref{DesS4}) belong to desired weighted spaces, ${\mathbb L}_{p,q,(\mu)}(Q)$ 
and $\widetilde{\mathbb L}_{p,q,(\mu)}(Q)$, respectively. 
By the same reasons, the requirements on
$\partial\Omega$ imply $\partial\Omega\in {\cal C}^1$ and ensure the
invariance of assumptions on $b^i$ under rectifying of the boundary.

Next, after rectifying of $\partial\Omega$ we can assume without loss of
generality that $\gamma^i(0)=\delta_i^n$ and rewrite the boundary condition
as follows:
\begin{equation}\label{inhom}
D_nu|_{x_n=0}=\varphi\equiv(\delta_i^n-\gamma^i(x,t))D_iu.
\end{equation}
The inhomogeneity in boundary condition (\ref{inhom}) will be removed if we subtract from $u$
some function satisfying the same boundary condition. By assumption 
$\gamma^i\in{\cal C}^{0,1;\frac 12}(\partial''Q)$, the function $\varphi$ 
has the same differential properties as $Du$. Therefore, such a subtraction leaves the space
${\mathbb L}_{p,q,(\mu)}(Q)$ (respectively, $\widetilde{\mathbb L}_{p,q,(\mu)}(Q)$) of the right-hand side in (\ref{DesS4}). 
This completes the proof.
\end{proof}

The assumption $\gamma^i\in{\cal C}^{0,1;\frac 12}(\partial''Q)$ is not optimal. The sharp assumption here
is that multiplication by the vector field $\gamma$ should keep the space of traces of gradients of functions 
from ${\mathbb W}^{2,1}_{p,q,(\mu)}(Q)$ (respectively, from $\widetilde{\mathbb W}^{2,1}_{p,q,(\mu)}(Q)$).
In other words, $\gamma$ should belong to space ${\bf MT}D{\mathbb W}^{2,1}_{p,q,(\mu)}(Q)$
(respectively, ${\bf MT}D\widetilde{\mathbb W}^{2,1}_{p,q,(\mu)}(Q)$)
of multipliers of traces of gradients of weighted Sobolev functions.
 
Unfortunately, to the best of our knowledge, these spaces are not described yet. In the isotropic case $p=q$
we can give rather sharp sufficient conditions in terms of the Besov spaces
(the notation of the Besov spaces corresponds to \cite[Ch.IV]{BIN}). The following result can be extracted from 
the proofs of \cite[Theorems 18.13 and 18.14]{BIN}, \cite{Usp} and \cite[4.4.3]{MSh}.


\begin{sats}\label{isotropic}
Let $1<p<\infty$ and $\mu \in
(-\frac{1}{p},1-\frac{1}{p})$.
\medskip

Suppose that $ b^i \in {\mathbb L}_{{\overline{p}},(\overline{\mu})}(Q)+ 
{\mathbb L}_{\infty,(\overline{\overline{\mu}})}(Q)$, 
where $\overline{p}$, $\overline {\mu}$ and $\overline{\overline{\mu}}$
 are subject to
\begin{equation*}
\begin{array}{ll}
\overline{p}=\max \{p, n+2\}, \quad\mbox{if}\quad p\ne n+2;\quad&
\overline{p}>n+2,\quad\mbox{if}\quad p=n+2;\\
\\
\overline{\mu}=\min\{\mu, \max\{1-\frac {n+2}{p},0\}\};&
\overline{\overline{\mu}}<\mu+\textstyle\frac 1p.
\end{array}
\end{equation*}
Suppose also that either $\partial\Omega \in {\cal W}^2_{\infty,(\overline{\overline{\mu}})}$ 
or $\partial\Omega \in {\cal W}^2_{\overline{p},(\overline{\mu})}$. 

Finally, we assume that the components $\gamma^i$ belong to the Besov space 
$B^{\mbox{\scriptsize\boldmath$\lambda$}}_{\overline{\overline{p}},\theta}(\partial''Q)$
with parameters
$$\gathered
\mbox{\boldmath$\lambda$}\equiv(\lambda_x^1,\dots,\lambda_x^{n-1},\lambda_t)=
\Big(1-\frac 1p,\dots,1-\frac 1p,\frac 12-\frac 1{2p}\Big);\qquad \theta=p;\\
\overline{\overline{p}}=\max \Big\{p, \frac {n+1}{1-\mu-\frac 1p}\Big\}, 
\quad\mbox{if}\quad p\ne \frac {n+2}{1-\mu};\quad
\overline{\overline{p}}>\frac {n+2}{1-\mu},\quad\mbox{if}\quad p=\frac {n+2}{1-\mu}.
\endgathered
$$
Then, for any $f \in{\mathbb L}_{p,(\mu)}(Q)$, the initial-boundary value problem
{\rm (\ref{DesS4})} has a unique solution $u\in {\mathbb W}^{2,1}_{p,(\mu)}(Q)$. 
Moreover, this solution satisfies
\begin{equation*}
{\pmb \|}\partial_tu{\pmb \|}_{p,(\mu)}+
\sum_{ij}{\pmb \|}D_iD_ju{\pmb \|}_{p,(\mu)} \le C {\pmb\|}f{\pmb \|}_{p,(\mu)},
\end{equation*}
where the positive constant $C$ does not depend on $f$.

\end{sats}

\bigskip

V.~K. was supported by the Swedish Research Council (VR). 
A.~N. was supported by RFBR grant 12-01-00439 and by St{.} Petersburg University grant. 
He also acknowledges the Link\"oping University for the financial support of his visit 
in February 2012.

\vspace{\baselineskip}

\end{document}